\documentclass[10pt]{amsart}

\usepackage[a4paper]{geometry}

\usepackage{graphicx}
\usepackage{amsmath,amsthm, amsfonts, amssymb}

\newtheorem*{theorem*}{Theorem}
\newtheorem{lemma}{Lemma}[section]
\newtheorem{proposition}[lemma]{Proposition}

\theoremstyle{definition}
\newtheorem{definition}[lemma]{Definition}
\theoremstyle{remark}
\newtheorem{remark}{Remark}[section]

\newcommand{\naturales}{\mathbb{N}}
\newcommand{\real} {\mathbb{R}}
\newcommand{\enteros} {\mathbb{Z}}

\newcommand{\R}{{\bf\sf R}}
\newcommand{\subR}{{\bf\sf Q}}

\newcommand{\dist}{\mbox{dist}}
\newcommand{\diam}{\mbox{diam}}

\newcommand{\dH}{\dim_{\mathcal H}\,}

\title[Variational principles and nonuniformly hyperbolic dynamics]{A variational principle for systems with \\ nonuniformly hyperbolic behavior with \\ 
applications to the dimension theory}

\author{ 
Fernando Jos\'e S\'anchez-Salas}
\address{
Departamento de Matem\'aticas, Facultad Experimental de Ciencias, Universidad del Zulia, Avenida Universidad, Edificio Grano de Oro, Maracaibo, Venezuela
}
\email{fjss@fec.luz.edu.ve}

\subjclass[2010]{37D25, 37D35}

\keywords{Nonuniformly hyperbolic systems, nonadditive thermodynamic formalism, fractal dimensions}

\thanks{ This work was partially supported by the Associateship Programme of ICTP}

\date{July 29, 2015}

\begin{document}

\begin{abstract}
Let $f$ be a $C^{1+\alpha}$ nonuniformly hyperbolic diffeomorphism. We use a a nonadditive version of the topological pressure of a 
class of admissible, possibly noncontinuous potentials $P^*(\Phi)$ to prove the following variational equation: 
$P^*(\Phi) = \sup_{\Omega \in {\mathcal H}}P^*(f|\Omega,\Phi)$ supremum taken over the set ${\mathcal H}$ of basic subsets in $M$. As a consequence we 
find a lower bound for the Cantor dimension of the stable and unstable Cantor sets of a non trivial conformal nonuniformly hyperbolic isolated sets.
\end{abstract}

\maketitle

In this note we use ideas from non additive thermodynamic formalism from \cite{barreira.2011} and methods of \cite{katok.mendoza} to approximate 
dimension-like quantities of the dynamics along suitable sequences of hyperbolic Cantor sets in systems with some hyperbolicity in the phase space without being 
uniformly hyperbolic. This type of questions has been considered previously by several authors. In 1984, A. Katok laid the foundations to 
study this type of questions in his seminal paper \cite{katok} about relations between entropy, periodic orbits and Lyapunov exponents of systems with 
nonuniformly hyperbolic behavior. We refer the reader to \cite{gelfert.wolf}, \cite{gelfert.2009}, \cite{gelfert.2010}, \cite{liang.liu.sun}
\cite{luzzatto.sanchez}, \cite{sanchez-salas.1}, \cite{sanchez-salas.2} and \cite{wang.sun} for some recent contributions to the subject.

As a consequence of our approch we prove the following
\ 
\\
\\
{\bf Theorem A} \ {\em Let $\Lambda$ be a compact, $f$-invariant, locally maximal, topologically transitive, nonuniformly hyperbolic subset of a 
conformal $C^{1+\alpha}$ diffeomorphism. Suppose in addition that $\Lambda$ is the support of an ergodic nonatomic hyperbolic measure $\mu$. Then,
\begin{equation}\label{main}
\dH(W^s(x) \cap \Lambda) \geq d^s \quad\text{and}\quad \dH(W^u(x) \cap \Lambda) \geq d^u
\end{equation}
where  $0 < d^s \leq \dim(E^s)$ (resp. $0 < d^u \leq \dim(E^u)$) is the unique solution to the \emph{Bowen equation},
\begin{equation}\label{bowen.equation.1}
\sup_{\nu \in \mathcal{M}_f(\Lambda)}\{h(\nu) - d^s\int\log\phi^sd\nu\} = 0,
\end{equation}
respectively,
\begin{equation}\label{bowen.equation.2}
\sup_{\nu \in \mathcal{M}_f(\Lambda)}\{h(\nu) - d^u\int\phi^ud\nu\} = 0,
\end{equation}
where $\phi^s(x) := -\log\|Df|E^s(x)\|$ and $\phi^u(x) := -\log\|Df|E^u(x)\|$ are the stable and unstable potentials.
}
\ 
\\
\\
Here $W^s(x)$ (resp. $W^u(x)$) are the stable (resp. unstable) manifold by $x$ and $\dH(X)$ is the Hausdorff dimension of a set $X \subset M$, putting in $M$ the distance defined by the Riemannian metric. We refer 
to \cite{falconer.1985}. The Hausdorff dimension of the stable (resp. unstable) Cantor sets $W^s(x) \cap \Lambda$ (resp. $W^u(x) \cap \Lambda$) 
give us a quantitative estimation of the size of $\Lambda$.

We recall that a $C^{1}$ diffeomorphism $f$ is called {\em conformal} is there exists a continuous positive function $a(x) > 0$ such that $Df(x) = a(x)I_x$ where 
$I_x : T_xM \to T_{f(x)}M$ is an isometry. 

In contrast with the nonuniformly and nonconformal case, the dimension theory of conformal uniformly hyperbolic sets is well understood in terms of 
the thermodynamics of these systems. 

Thermodynamic formalism is the primary source of variational principles in dynamical systems. Its main ingredientes are the 
{\em topological pressure} $P(\phi)$ of a continuous potential $\phi$, a variational principle and equilibrium states. Pressure is a 
topological invariant of the dynamics introduced by D. Ruelle in \cite{ruelle.1973} for a class of $\enteros^n$ actions arising naturally in the 
formalism of equilibrium statistical physics and later extended for continuous maps $f$ of compact metric space in \cite{walters} who proved 
the following well-known \emph{variational principle for the topological pressure}: {\em Let $f$ be a continuous transformation of a compact metric space and $\phi$ continuous. Then, 
\begin{equation}\label{additive.variational.principle}
P(\phi) = \sup_{\mu \in {\mathcal M}_f}\left\{h(\mu) + \int\phi{d\mu}\right\}. 
\end{equation}
}
$h(\mu)$ is the {\em Kolmogorov-Sinai entropy} of an invariant Borel probability $\mu$ and supremum is taken over ${\mathcal M}_f$, 
the set of $f$-invariant Borel probabilities endowed with the weak topology. We call $P_{\mu}(\phi) := h(\mu) + \int\phi{d\mu}$ the {\em free energy or 
measure-theoretical pressure}. A Borel probability $\mu$ is called an {\em equilibrium state}  if $P_{\mu}(\phi)$ attains its maximum, i.e. 
$P(\phi) =  h(\mu) + \int\phi{d\mu}$. Variational principle (\ref{additive.variational.principle}) generalizes for topological pressure a similar 
variational property of the entropy due to Dinaburg. See \cite{manhe}. 

Existence and uniqueness of equilibrium states depends on properties of the dynamics and regularity of the potentials. These notions are well understood 
for Axiom A systems. Many important ergodic properties of uniformly hyperbolic dynamical systems such that the existence of Sinai-Bowen-Ruelle (SRB) 
measures, measures of maximal entropy, computation of rates of escape and dimension-like quantities of the dynamics are elaborated upon these notions. 
See \cite{bowen}.

The following is central result of the dimension theory of uniformly hyperbolic sets.
\ 
\\
\\
{\bf Bowen's equation} \ {\em Let $\Lambda$ be a compact $f$-invariant uniformly hyperbolic isolated and topologically mixing set of a $C^{1}$ conformal 
diffeomorphism $f$. Then 
$$
\dim_{\mathcal{H}}(W^s(x) \cap \Omega) = d^s \quad\text{and}\quad \dim_{\mathcal{H}}(W^u(x) \cap \Omega) = d^u,
$$
where $d^s$ (resp. $d^u$) is the unique solution to the \emph{Bowen equation} $P(f|\Lambda,-d^s\phi^s) = 0$ (resp. $P(f|\Lambda,-d^u\phi^u) = 0$). See \cite[Theorem 6.2.8]{barreira.2011}.
}
\ 
\\

Starting from Bowen's equation one may develop dimension theory of dynamical systems either beyond the conformal setting or studying nonuniformly 
hyperbolic sets which is the approch that we follow in this note. However, as long as $\Lambda$ is nonuniformly hyperbolic the stable and unstable 
potentials are just Borel measurable, hence the left hand side of (\ref{bowen.equation.1}) and (\ref{bowen.equation.2}) is not the topological pressure. 
This move us to develop thermodynamics formalism beyond the realm of uniformly hyperbolic dynamics. For this one need, as a first step, to introduce a 
new concept of pressure allowed to consider Borel measurable potentials. Moreover, new methods are needed to establish the existence and uniqueness of 
equilibrium states. See \cite{pesin}.

An earlier temptative to extend the notion of topological pressure was made by Falconer motivated by the application of thermodynamics to the study of 
fractal dimensions for non-conformal transformations. In his work \cite{falconer} he introduced a new notion of topological 
pressure $P({\mathcal F})$ for subadditive sequences ${\mathcal F} = \{\phi_n\}$ of continuous functions and an analog to Bowen's equation, proving a 
variational principle similar to (\ref{additive.variational.principle}) supposing some regularity properties of the sequence ${\mathcal F}$

We recall that a sequence ${\mathcal F} = \{\phi_n\}$ of continuous functions is {\em subadditive (resp. superadditive)} if $\phi_{n+m} \leq \phi_n + \phi_m \circ f^n$ 
(resp. $\phi_{n+m} \geq \phi_n + \phi_m \circ f^n$). 

Later, Barreira extended in \cite{barreira} previous results of Pesin and Pitskel \cite{pesin.pitskel} defining topological pressure as a 
dimensional-like quantity by Carath\'eodory's method. Lets recall the definitions.

Let $\mathcal{U}$ be a finite covering. We denote $\mathcal{W}_n(\mathcal{U})$ the set of sequences $U = (U_0 \cdots U_{n-1})$ of length $n > 0$ of open 
sets in $\mathcal{U}$. Given $U \in \mathcal{W}_n(\mathcal{U})$ we define $m(U) = n$ and $X(U) = \bigcap_{k=0}^{n-1}f^{-k}U_k$. A collection 
$\Gamma \subset \bigcup_{n \in \naturales}\mathcal{W}_n(\mathcal{U})$ covers $Z \subset M$ if $Z \subset \bigcup_{U \in \Gamma}X(U)$. Given a sequence 
$\mathcal{F} = \{\phi_n\}$ for each $n \in \naturales$ we define 
$$
\gamma_n(\mathcal{F}, \mathcal{U}) = \sup\{|\phi_n(x) - \phi_n(y)|: x,y \in X(U), \ \text{for some}\ U \in \mathcal{W}_n(\mathcal{U})\}
$$
We say that $\mathcal{F}$ has \emph{tempered variation} if
$$
\limsup_{\diam(\mathcal{U}) \to 0^+}\limsup_{n \to +\infty}\dfrac{\gamma_n(\mathcal{F}, \mathcal{U})}{n} = 0.
$$
For each $n \in \naturales$ and $U \in \mathcal{W}_n(\mathcal{U})$ we write
$$
\phi(U) = 
\begin{cases}
\sup_{x \in X(U)}\phi_n(x) & \text{if} X(U) \not= \emptyset\\
-\infty,                   & \text{otherwise}
\end{cases}
$$
Given $Z \subset M$ and $a \in \real$ we define
$$
M_{Z}(a,\mathcal{F}, \mathcal{U}) = \lim_{n \to +\infty}\inf_{\Gamma}\sum_{U \in \Gamma}\exp(-m(U) + \phi(U)),
$$
where infimum is taken over the set of coverings $\Gamma \subset \bigcup_{n \in \naturales}\mathcal{W}_n(\mathcal{U})$ of $Z$. We also define,
$$
\underline{M}_{Z}(a,\mathcal{F}, \mathcal{U}) = \liminf_{n \to +\infty}\inf_{\Gamma}\sum_{U \in \Gamma}\exp(-m(U) + \phi(U)),
$$
and
$$
\overline{M}_{Z}(a,\mathcal{F}, \mathcal{U}) = \limsup_{n \to +\infty}\inf_{\Gamma}\sum_{U \in \Gamma}\exp(-m(U) + \phi(U)).
$$
Then one prove that when $a$ goes from $-\infty$ to $+\infty$ the quantities so defined jump from $+\infty$ to $0$ at a unique value so we can define
\begin{eqnarray*}
P_Z(\mathcal{F},\mathcal{U}) = \inf\{a : M_{Z}(a,\mathcal{F}, \mathcal{U}) = 0 \ \}\\
\underline{P}_Z(\mathcal{F},\mathcal{U}) = \inf\{a : \underline{M}_{Z}(a,\mathcal{F}, \mathcal{U}) = 0 \ \} \\
\overline{P}_Z(\mathcal{F},\mathcal{U}) = \inf\{a : \overline{M}_{Z}(a,\mathcal{F}, \mathcal{U}) = 0 \ \}
\end{eqnarray*}
Then is proved that, if $\mathcal{F}$ has tempered variation then 
$$
P_Z(\mathcal{F}) = \lim_{\diam(\mathcal{U}) \to 0^+}P_Z(\mathcal{F},\mathcal{U})
$$
(resp. $\underline{P}_Z(\mathcal{F})$ and $\overline{P}_Z(\mathcal{F})$) is well-defined (\cite[Theorem 4.1.2]{barreira.2011}). If $\mathcal{F} = \{\sum_{k=0}^{n-1}\phi(f^k(x))\}$ and 
$\Lambda \subset M$ is a compact $f$-invariant subset then $P_{\Lambda}(\mathcal{F}) = P(f|\Lambda, \phi)$ (\cite[pp. 59]{barreira.2011}) and also contain 
as a particular case notion of topological pressure of a subadditive sequence introduced by Falconer in \cite{falconer}. 

A variational principle similar to (\ref{additive.variational.principle}) was established for this nonadditive pressure. Namely it is proved 
in \cite[Theorem 4.3.1]{barreira.2011} that for every continuous selfmap $f$ of a compact metric space $X$ and for every Borel measurable 
$f$-invariant set, if $\mathcal{F} = \{\phi_n\}$ is a sequence of continuous functions with tempered variation and if 
there exists a continuous function $\psi$ such that $\phi_{n+1} - \phi_n \circ f \to \psi$ converges uniformly then
$$
P_{\mathcal{L}(Z)}(\mathcal{F}) = \sup_{\mu \in \mathcal{M}_f(Z)}\{h(\mu) + \int\psi{d\mu}\},
$$
where $\mathcal{L}(Z) = \{x \in Z: 1/n\sum_{k=0}^{n-1}\delta_{f^k(x)} \ \text{has a subsequence convergent to some} \ \mu \in \mathcal{M}_f(Z)\}$.

The nonadditive thermodynamic formalism had been used succesfully to give useful estimates of the Hausdorff dimension and topological capacity of 
nonconformal compact $f$-invariant subsets. See for instance \cite[Chapter 5, Chapter 6]{barreira.2011} and \cite{barreira.gelfert} for a recent survey of applications from 
thermodynamics to the dimension theory of dynamical systems.

More recently Cao, Feng and Huang \cite{cao.feng.huang} provided a proof of a \emph{subadditive variational principle} by using the following

\begin{definition}
Let $f : X \to X$ a continuous selfmap of a complete metric space $(X,d)$ and let ${\mathcal F} = \{\phi_n\}$ be a sequence of continuous real functions. 
We define the {\em nonadditive topological pressure of ${\mathcal F}$} as
\begin{equation}\label{nonadditive.topological.pressure}
P({\mathcal F}) := \lim_{\epsilon \to 0^+}\lim_{n \to +\infty}\dfrac{1}{n}\log\left(\inf_E\left\{\sum_{x \in E}\exp{\phi_n(x)}\right\}\right),
\end{equation}
\emph{infimum taken over $(\epsilon,n)$-spanning subsets $E \subset M$}, where we recall that $E \subset X$ is {\em $(\epsilon,n)$-spanning set in $X$} 
if for every $x \in X$ there exists $y \in E$ such that $d(f^k(x),f^k(y)) \leq \epsilon$, for every $0 \leq k \leq n-1$. 
\end{definition}

If we let $\phi$ be continuous and define $S_n\phi(x) := \sum_{j=0}^{n-1}\phi(f^j(x))$, then $P(\{S_n\phi\}) = P(\phi)$ 
is the (additive) topological pressure. $P(\mathcal{F})$ definition is equivalent to Falconer's approach for a mixing repeller and it is equal 
to $P_{M}(\mathcal{F})$, the Barreira dimension-like definition of nonadditive topological pressure, under the additional assumption that 
${\mathcal F} = \{\phi_n\}$ has tempered variation. See \cite[Proposition 4.7]{cao.feng.huang} and \cite[Chapter 7]{barreira.2011}.
\ 
\\
\\
{\bf Subadditive variational principle} \ {\em Let $f : X \to X$ a continuous self map of a compact metric space $(X,d)$ and ${\mathcal F} = \{\phi_n\}$ a subadditive sequence of continuous functions. 
Suppose in addition that the rate of growing is uniformly bounded from below
$$
\Phi = \inf\limits_{n > 0}\dfrac{\phi_n}{n} > -\infty.
$$
Then,
\begin{equation}\label{subadditive.variational.principle}
P({\mathcal F}) = \sup\limits_{\mu \in {\mathcal M}_f}\left\{h(\mu) + \int\Phi{d\mu}\right\}.
\end{equation}
}
See \cite[Theorem 1]{cao.feng.huang}. A major virtue of the above result is that it don't require any additional assumptions on the regularity of the 
family $\mathcal{F} = \{\phi_n\}$ except to be subadditive.

\subsection{Statement of main results}

Motivated by the subadditive variational principle we define the \emph{variational pressure} of a Borel measurable potential $\Phi$,
\begin{equation}\label{variational.pressure}
P^*(\Phi) := \sup\left\{h(\mu) + \int\Phi{d\mu}: \mu \in {\mathcal M}_f\right\}.
\end{equation}
For $P^*(\Phi)$ to make sense it is necessary $\Phi$ to be integrable with respect to (w.r.t.) every $f$-invariant Borel probability 
$\mu \in {\mathcal M}_f$. For this we introduce the following class of admissible (possibly discontinuous) potentials $\Phi$.

\begin{definition}
We say that a Borel measurable real function $\Phi$ is an \emph{admissible potential} if either it is a continuous or if it is the rate of growing 
of a sub(super)additive sequence of continuous functions $\{\phi_n\}$. We require in addition that $\|\Phi\|_{\infty,\mu} < +\infty$ for every $f$-invariant 
Borel probability $\mu \in {\mathcal M}_f$, where $\|\Phi\|_{\infty,\mu}$ is the $\mu$-essential supremum. We denote by ${\mathcal S}^+$ the set of admissible 
potentials.
\end{definition}

It follows from Kingman's subadditive ergodic theorem (see \cite{ruelle.1979}) that every admissible potential is $\mu$-integrable, for every 
$f$-invariant Borel probability $\mu \in \mathcal{M}_f$ and therefore $P^*(\Phi)$ makes sense for every $\Phi \in \mathcal{S}^+$.

If $\Phi$ is the rate of growing of a subadditive sequence of continuous functions $\mathcal{F} = \{\phi_n\}$ then $P^*(\Phi) = P(\mathcal{F})$ is 
the subadditive topological pressure (\ref{nonadditive.topological.pressure}). On the other hand, if $\Phi$ is the rate of growing of a superadditive sequence $\{\phi_n\}$ then, as we shall prove
below,
$$
P^*(\Phi) = \sup_{n > 0}P\left(\dfrac{\phi_n}{n}\right).
$$
\begin{remark}
For the sake of completeness in the theory it would be desirable to extend the subadditive variational principle 
(\ref{subadditive.variational.principle}) for superadditive potentials $\mathcal{F} = \{\phi_n\}$. However we don't need that to prove the present 
results. The point is that the variational pressure of a superadditive potential is a well-defined topological invariant of the dynamics, which is the 
case by the above formula.
\end{remark}

Our idea is to extend the nonuniformly hyperbolic variational principle
\begin{equation}\label{nouniformily.hyperbolic.variational.principle.1}
P(\phi) = \sup_{\Omega \in \mathcal{H}}P(f|\Omega,\phi)
\end{equation}
for the variational pressure $P^*(\Phi)$ of an admissible potential $\Phi$, where supremum is taken over family of basic sets, that is, \emph{compact, 
$f$-invariant, locally maximal, topologically transitive, uniformly hyperbolic sets}. That is, we are looking for to give sufficient conditions on 
$\Phi$ or the dynamics for to have
\begin{equation}\label{nouniformily.hyperbolic.variational.principle.2}
P^*(\Phi) = \sup_{\Omega \in \mathcal{H}}P(f|\Omega,\Phi). 
\end{equation}

Variational equation (\ref{nouniformily.hyperbolic.variational.principle.1}) has been considered in the work of Barreira and Iommi 
\cite[Theorem 5]{barreira.iommi} for certain class of dynamical systems where hyperbolic measures are dense is ome sense. However, when every $f$-invariant ergodic 
Borel probability is hyperbolic, one may have continuous potentials for which $P(\phi) > \sup_{\Omega \in \mathcal{H}}P(f|\Omega,\phi)$. See 
\cite[Example 1.1]{sanchez-salas.2015} and \cite{barreira.iommi}.

To prove (\ref{nouniformily.hyperbolic.variational.principle.2}) we start showing that given an admissible potential $\Phi \in \mathcal{S}^+$ and a 
nonatomic hyperbolic $f$-invariant ergodic Borel probability $\mu$ there exists a sequence of basic sets $\Omega_n$ such that
\begin{equation}\label{approximating.free.energy}
P^*(f|\Omega_n,\phi) \to h(\mu) + \int\Phi{d\mu}. 
\end{equation}
We then introduce a class of continuous potentials $\Phi$ for which there exists a sequence of hyperbolic nonatomic measures $\mu_n$ with $h(\mu_n) > 0$ 
such that $h(\mu_n) + \int\Phi{d\mu_n} \to P(\Phi)$ and then use a 'diagonal' argument to get a sequence of basic sets $\Omega_n$ such that 
$P^*(f|\Omega_n,\Phi) \to P^*(\Phi)$.
\ 
\\
\\
{\bf Theorem B} \ {\em Let $f$ be a {\em regular nonuniformly hyperbolic diffeomorphism} of a compact Riemannian manifold, $\mu$ be a hyperbolic ergodic Borel 
probability with positive metrical entropy and $\Phi \in {\mathcal S}^+$ {\bf an admissible potential with tempered variation}. Then there exists a sequence 
of basic sets $\Omega_n$ and a constant $\chi > 0$ such that:
\begin{enumerate}
\item[a)] the rate of hyperbolicity of $\Omega_n$ is bounded from below by $\chi > 0$;
\item[b)] $\mu_n \to \mu$ {\bf for every sequence of ergodic measures} with $supp(\mu_n) \subseteq \Omega_n$;
\item[c)] ${P^*}(f|\Omega_n,\Phi) \to h(\mu) + \int\Phi{d\mu}$.
\end{enumerate}
}
Most of the paper will be dedicated to prove Theorem B.

Next definition introduce a class of admissible potentials which generalizes hyperbolic potentials used in \cite{sanchez-salas.2015} to extend the 
nonuniformly hyperbolic variational principle (\ref{nouniformily.hyperbolic.variational.principle.1}) to the present setting.

\begin{definition}
Let $\Phi \in {\mathcal S}^+$ be an admissible potential. We say that $\Phi$ is \emph{hyperbolic} if
\begin{equation}\label{hyperbolic.potential}
 P^*(\Phi) - \sup_{\mu \in \mathcal{M}_f}\int\Phi{d\mu} > 0.
\end{equation}
\end{definition}
From Theorem A and the definition of hyperbolic admissible potential we get the
\ 
\\
\\
{\bf Corollary C}
{\em Let $f$ be a {\em nonuniformly hyperbolic} $C^{1+\alpha}$ diffeomorphism of a compact Riemannian manifold and let $\Phi \in {\mathcal S}^+$  be an 
admissible potential with {\em tempered variation}. Suppose in addition that $\Phi$ is hyperbolic. Then 
\begin{equation}\label{main.thm.1}
P^*(\Phi) = \sup\limits_{\Omega \in {\mathcal H}} \ P^*(f | \Omega,\Phi),
\end{equation}
where ${\mathcal H}$ is the family of basic sets $\Omega \subset M$. 
}
\ 
\\
\\
We notice that, being $\Phi$ subadditive and of tempered variation then $P^*(\Phi)$ coincides with Barreira's dimension-like definition of 
nonadditive pressure and then Corollary can be used as a tool to extend for the case of nonuniformly hyperbolic sets some of the estimates 
of dimension of nonconformal uniformly hyperbolic sets exposed in \cite[Chapter 5, Chapter 6]{barreira.2011}.

\begin{proof}[Proof of Corollary C]
If $\Phi$ is hyperbolic and $\mu_n$ 
a sequence of hyperbolic measures such that
$$
P_{\mu_n}(\Phi) = h(\mu_n) + \int\Phi{d\mu_n} \to P^*(\Phi),
$$
then for all $n >> 1$ sufficiently large $h(\mu_n) > 0$. Indeed, taking $0 < \epsilon < P^*(\Phi) - \sup_{\mu \in \mathcal{M}_f}\int\Phi{d\mu}$ and then 
$N > 0$ such that 
$$
h(\mu_n) + \int\Phi{d\mu_n} > P^*(\Phi) - \epsilon > 0, \quad \forall \ n \geq N,
$$
then
$$
h(\mu_n) > P^*(\Phi) - \int\Phi{d\mu_n} - \epsilon \geq P^*(\Phi) - \sup_{\mu \in \mathcal{M}_f}\int\Phi{d\mu} - \epsilon > 0, \quad \forall \ n \geq N.
$$
Therefore, by Theorem A, for every such $\mu_n$ there exists a sequence $\Omega^n_m$ of basic sets such that
$$
P^*(f | \Omega^n_m,\Phi) \to h(\mu_n) + \int\Phi{d\mu_n}.
$$
Then, taking a suitable diagonal sequence $\Omega_n = \Omega^n_{m_n}$ we get that
$$
P^*(f | \Omega_n,\Phi) \to P^*(\Phi),
$$
concluding that $P^*(\Phi) = \sup_{\Omega \in \mathcal{H}}P^*(f|\Omega,\Phi)$, so proving Corollary B. 
\end{proof}

However, to prove Theorem A we need to consider admissible potentials which have not tempered variation. In order to do this we introduce a class of 
nonuniformly hyperbolic systems where basic sets are dense.

\begin{definition}
Let $f$ be a nonuniformly hyperbolic $C^{1+\alpha}$ diffeomorphism. We say that $f$ admits a \emph{hyperbolic exhaustion} if there exists an 
increasing sequence of basic sets $\Omega_n \subset M$ such that:
\begin{equation}\label{regular.nonuniformly.hyperbolic}
 M = \overline{\bigcup_n\Omega_n}
\end{equation}
\end{definition}
Notice that if $f$ admits a hyperbolic exhaustion then it does not have isolated hyperbolic periodic orbits. It follows easily from \cite{luzzatto.sanchez} 
that if $f$ is nonuniformly hyperbolic and there exists a nonatomic ergodic $f$-invariant measure such that $M = supp \  \mu$, then $f$ admits a hyperbolic 
exhaustion. If $f$ is nonuniformly hyperbolic and $M$ admits a hyperbolic exhaustion then one easily prove that 
(\ref{nouniformily.hyperbolic.variational.principle.1}) holds for every continuous potential. This permits to prove the following
\ 
\\
\\
{\bf Theorem D}
{\em Let $f$ be a nonuniformly hyperbolic $C^{1+\alpha}$ diffeomorphism admitting a hyperbolic exhaustion and let $\Phi \in {\mathcal S}^+$ be the rate of 
growing of a superadditive sequence $\{\phi_n\}$ then: 
\begin{equation}\label{main.thm.3}
P^*(\Phi) = \sup\limits_{\Omega \in {\mathcal H}}\,P^*(f|\Omega,\Phi).
\end{equation} 
}
\ 
\\
Notice that we do not require $\Phi$ to be of tempered variation neither do we approximate the measure-theoretical pressure 
$P^*_{\mu}(\Phi)$ by basic sets as we did in Theorem B.

\section{Proof of main results}

\begin{proof}[Proof of Theorem A using Theorem D]
We take the case of the unstable dimension, since the stable is similar. Let 
$$
\Phi^u(x) = \lim_{n \to +\infty}-\dfrac{1}{n}\log|\bigwedge^{\dim(M)}Df^{n}(x)| = \sup_{n > 0}-\dfrac{1}{n}\log|\bigwedge^{\dim(M)}Df^{n}(x)|.
$$
For the stable dimension we take
$$
\Phi^s(x) = \lim_{n \to +\infty}-\dfrac{1}{n}\log|\bigwedge^{\dim(M)}Df^{-n}(x)| = \sup_{n > 0}-\dfrac{1}{n}\log|\bigwedge^{\dim(M)}Df^{-n}(x)|.
$$
Then $\Phi^u$ (resp. $\Phi^s$) is an admissible superadditive potential since $\phi_n = -\log|\bigwedge^{m}Df^{n}(x)|$ is a superadditive sequence of 
continuous functions. Moreover, 
$$
\int\Phi^u{d\mu} = \int\log{J^uf}{d\mu}
$$
for every $f$-invariant Borel probability $\mu \in \mathcal{M}_f$, by the Oseledec theorem, where $J^uf(x) = |\det(Df|E^u(x))|$. Similarly so
$$
\int\Phi^s{d\mu} = \int\log{J^s}{d\mu}
$$
Then,
$$
P^*(f|\Lambda,-d\Phi^u) = P^*(f|\Lambda,-d\log{J^u}) \quad\text{(resp.}\quad P^*(f|\Lambda,-d\Phi^s) = P^*(f|\Lambda,-d\log{J^s}). 
$$
Then, by Theorem D,
$$
P^*(f|\Lambda,-d^u\log{J^u}) = \sup\limits_{\substack{\Omega \in \mathcal{H} \\ \Omega \subset \Lambda}}P^*(f|\Omega,-d\log{J^u}) = 0.
$$
This implies that there exists a sequence of hyperbolic basic sets $\Lambda_n \subset \Lambda$ such that
$$
P^*(f|\Lambda_n,-d^u\log{J^u}) \to 0
$$
and similarly for $P^*(f|\Lambda_n,-d^s\log{J^s}) = 0$.

Notice that we may suppose that $\Lambda_n \subset \Lambda_{n+1}$. Now, we use the following

\begin{lemma}
Let $\Lambda$ be a basic set for a $C^{1+\alpha}$ conformal diffeomorphism of a compact manifold $M$. Then
\begin{equation}\label{hausdorff.dimension}
 \dH(\Lambda \cap W^u(x)) = d\dim(E^u)
\end{equation}
where $0 < d \leq 1$ is the unique solution to the equation
\begin{equation}\label{bowen.equation.3}
P(f|\Lambda, -d\log(J^uf)) = 0.
\end{equation}
\end{lemma}
\begin{proof}
Let $Df(x) = a(x)I_x$ where $I_x: T_xM \to T_f(x)M$ is an isometry. Then, 
$$
J^uf(x) = |\det(Df|E^u(x))| = \dim(E^u)|a(x)| = \dim(E^u)\|Df(x)\|.
$$ 
\end{proof}

By the convexity of $P^*(f|\Lambda_n,-t\log{J^u})$ for every $\Lambda_n$ and 
for $t \in \real$, we get that $d^u_n \uparrow d^u$, where $d^u_n$ is the unique solution to Bowen's equation for $\Lambda_n$, 
$P^*(f|\Lambda_n,-d^u_n\log{J^u}) = 0$. Therefore, $\dH(\Lambda_n \cap W^u(x)) \uparrow d^u\dim(E^u)$. But then as 
$\bigcup_n\Lambda_n \subset \Lambda$ we get that
$$
\dH(\Lambda \cap W^u(x)) \geq \sup_n\dH(\Lambda_n \cap W^u(x)) = d^u\dim(E^u).
$$
\end{proof}

\begin{proof}[Proof of Theorem D]
Observe that for every continuous potential $\phi$ and for every ergodic nonatomic hyperbolic measure $\mu$, 
there exists a sequence of basic sets $\Omega_n$ such that
$$
P(f|\Omega_n,\phi) \to h(\mu) + \int\phi{d\mu}.
$$
This is \cite[Theorem A]{sanchez-salas.2015} or, if you prefer, a direct consequence of Theorem A stated above, since continuous functions are 
admissible and $P^*(\Phi) = P(\phi)$, when $\Phi = \phi$.

\begin{lemma}
Let $\phi$ be continuous and $f : X \to X$ be a continuous selfmap of a compact metric space. Then $\Omega  \mapsto P(f|\Omega,\phi)$ is continuous when 
$\Omega$ varies on the family of compact $f$-invariant subsets with $P(f|\Omega,\phi) < +\infty$.
\end{lemma}

See \cite[Lemma 1.3]{sanchez-salas.2015}. Then, as $M$ admits a hyperbolic exhaustion then, for every continuous potential $\phi$,
$$
P(\phi) = \sup_{\Omega \in \mathcal{H}}P(f|\Omega,\phi),
$$
since there exists a sequence of basic sets $\Omega_n \uparrow M$ and therefore $P(f|\Omega_n,\phi) \to P(\phi)$, by continuity. See 
\cite[Proposition 1.2]{sanchez-salas.2015.2}. Now, we use the following

\begin{proposition}\label{limit.pressure.lemma}
Let $\{\phi_n\}$ be a sub(super)additive sequence of continuous functions, $\Phi \in {\mathcal S}^+$ its rate of growing. Then,
\begin{equation}\label{limit.pressure.sub}
P^*(\Phi) = \lim\limits_{n \to +\infty}P\left(\dfrac{\phi_n}{n}\right).
\end{equation}
Moreover, if $\{\phi_n\}$ is subadditive (resp. superadditive) then,
$$
P^*(\Phi) = \inf\limits_{n > 0}P\left(\dfrac{\phi_n}{n}\right) \quad\text{(resp. '$\sup$')}
$$
\end{proposition}

We prove proposition \ref{proof.limit.pressure.lemma} in section \ref{proof.limit.pressure.lemma}. We notice that this had been proved in 
\cite[Theorem 7.3.1]{barreira.2011} under the additional assumption that $\mu \mapsto h(\mu)$ is uppersemicontinuous.

Then given a superadditive sequence of continuous functions $\mathcal{F} = \{\phi_n\}$,
$$
P\left(\dfrac{\phi_n}{n}\right) = \sup_{\Omega \in \mathcal{H}}P\left(f|\Omega,\dfrac{\phi_n}{n}\right), \quad \ \forall \ n > 0.
$$
Then, taking supremun on $n > 0$ at both sides of the equation, we get
$$
P^*(\Phi) = \sup_{\Omega \in \mathcal{H}}P^*(f|\Omega, \Phi), 
$$
by proposition \ref{limit.pressure.lemma}. This conclude the proof of Theorem C. 
\end{proof}

\section{The proof of Theorem B}\label{sec:proof.main.technical.lemma}

Let $\rho > 0$ a small positive number and $s > 0$ an integer, $\{\psi_i\} \subset C(M)$ a countable dense subset of continuous 
functions and $\mu$ an ergodic non atomic hyperbolic Borel probability. Let $\{\phi_n\}$ be a sub(super)additive sequence of continuous functions such that
there exists $L > 0$ such that
\begin{equation}\label{bounds.subadditive.potential.1}
\dfrac{|\phi_n(x)|}{n} \leq L, \quad\forall \ x \in M \quad\text{and}\quad \forall \ n > 0.
\end{equation}
By Kingman's theorem, there exists a measurable $\Phi = \Phi(x)$ such that
\begin{equation}\label{definition.Phi}
\Phi(x) = \inf_{n > 0}\dfrac{\phi_n(x)}{n} \quad (\text{resp.}\Phi(x) = \sup_{n > 0}\dfrac{\phi_n(x)}{n})
\end{equation}
for $\mu$-a.e. for every $\mu \in {\mathcal M}_f$.
We suppose in addition that $\{\phi_n\}$ has \emph{tempered variation}.
\begin{definition}
Let $x$ be an Oseledec regular point. We recall that $x$ is hyperbolic if all the Lyapunov exponents at $x$ are non zero. The {\em rate of hyperbolicity of a hyperbolic regular point $x$} is defined 
as $\chi(x) := \min\{|\chi_i(x)|\}$, where $-\infty < \chi_1(x) < \cdots < \chi_k(x) < +\infty$ is the spectrum of Lyapunov exponents of $x$. See \cite{barreira.pesin}. 
We define the {\em rate of hyperbolicity of an $f$-invariant set $\Omega$} as 
$$
\chi(\Omega) := \inf\limits_{x \in \Omega}|\chi(x)|
$$
and the {\em rate of hyperbolicity of a measure $\mu$} as the infimum of $\chi(\Lambda)$ taken over the family of compact $f$-invariant 
subsets $\Lambda$ with $\mu(\Lambda) > 0$.
\end{definition}

Then we have the following
\begin{proposition}\label{main.technical.lemma}
There exists a continuous function $\Phi_{\rho}$ such that
\begin{equation}\label{convergence.in.measure}
\Phi_{\rho} \to \Phi \quad\text{in measure as}\quad \rho \to 0^+
\end{equation}
and a hyperbolic basic set  
$$
\Omega = \Omega(\rho,s,\Phi_{\rho})
$$
with rate of hyperbolicity bounded from below by a constant $\chi > 0$ such that:
\begin{enumerate}
\item[a)] every ergodic measure $\nu$ supported on $\Omega$ belongs to the basic weak-* open neighborhood ${\mathcal O}(\rho,s)$
$$
{\mathcal O}(\rho, s) := \{\nu : \left|\int\psi_i{d{\mu}} - \int\psi_i{d{\nu}}\right| < \rho, \ i = 1, \cdots , s\};
$$
\item[b)] there exists a subsequence $\mathcal{M}_0 = \{m_k\}_{k > 0}$ such that
\begin{equation}\label{main.1}
\dfrac{P^*_{\mu}(\Phi)-o(1)}{1+\rho} \leq P\left(f|\Omega,\dfrac{\phi_m}{m}\right) \leq P^*_{\mu}(\Phi) + o(1), \quad\forall \ m \in \mathcal{M}_0,
\end{equation}
where $o(1)$ is a positive function such that $o(1) \to 0^+$ when $\rho \to 0^+$ and 
\begin{equation}\label{free.energy.1}
P^*_{\mu}(\Phi) := h(\mu) + \int\Phi{d\mu}.
\end{equation}
\end{enumerate} 
\end{proposition}

Theorem B follows from Proposition \ref{main.technical.lemma} and proposition \ref{limit.pressure.lemma}.

\begin{proof}[Proof of Theorem B]
By (\ref{main.1})
\begin{eqnarray*}\label{main.proof.3}
\dfrac{P^*_{\mu}(\Phi)-o(1)}{1+\rho} & \leq & \lim_{k \to +\infty}P\left(f|\Omega,\dfrac{\phi_{m_k}}{m_k}\right)\\
							      & = & P^*(f|\Omega,\Phi).                                                                                                                      
\end{eqnarray*}
Moreover,
\begin{eqnarray*}
P^*(f|\Omega,\Phi) & = & \lim_{k \to +\infty}P\left(f|\Omega,\dfrac{\phi_{m_k}}{m_k}\right)\\
               &  \leq & P^*_{\mu}(\Phi) + o(1).
\end{eqnarray*}
Therefore,
\begin{equation}\label{main.proof.4}
\dfrac{\rho{P^*_{\mu}(\Phi)}-o(1)}{1+\rho} \leq P^*(f | \Omega,\Phi) - P^*_{\mu}(\Phi) \leq o(1).
\end{equation}
Now choose sequences $\rho_n \downarrow 0^+$, $s_n \to +\infty$ and $\Phi_n = \Phi_{\rho(n)}$ and define
\begin{equation}\label{main.proof.6}
\Omega_n = \Omega(\rho_n,s_n,\Phi_n).
\end{equation}
By (\ref{main.proof.4}), $\Omega_n$ is a sequence of hyperbolic basic sets satisfying the claims (1), (2) and (3) of Theorem A.
\end{proof}

\section{Proof of Proposition \ref{main.technical.lemma}: constructing $\Omega$}

Our starting point will be the description of the free energy
$$
P_{\mu}(\phi) = h(\mu) + \int\phi{d\mu}
$$
of a continuous function $\phi$ as a weighted rate of growing of dynamically non-equivalent finite orbits up to finite precision. For this we let $\mu$ 
an $f$-invariant Borel probability and define
\begin{equation}\label{definition.free.energy.1}
P_{\mu}(\phi) := \lim_{\alpha \to 0^+}\lim_{\epsilon \to 0^+}\lim_{n \to +\infty}\dfrac{1}{n}\log\left(\inf_E\left\{\sum_{x \in E}\exp{S_n\phi(x)}\right\}\right),
\end{equation}
\emph{infimum taken over $(\epsilon,n,\alpha)$-spanning subsets $E \subset M$}, where by $(\epsilon,n,\alpha)$-spanning we mean a finite subset 
$E \subset M$ such that
$$
\mu\left(\bigcup_{x \in E}B(x,\epsilon,n)\right) \geq \alpha,
$$
where
$$
B(x,\epsilon,n) := \{\,y \in X \,:\, \dist(f^j(x),f^j(y)) < \epsilon, \ j = 0, \cdots , n-1\,\}.
$$

The next proposition was proved in \cite{mendoza.1988}[Theorem 1.1].

\begin{proposition}\label{free.energy.prop}
Let $f : X \to X$ a continuous self map of a compact metric space $(X,d)$, $\phi$ continuous and $\mu \in {\mathcal M}_f$ an ergodic 
$f$-invariant Borel probability. Then, for every $\alpha > 0$,
\begin{equation}\label{definition.free.energy.2}
P_{\mu}(\phi) = \lim_{\epsilon \to 0^+}\lim_{n \to +\infty}\dfrac{1}{n}\log\left(\inf_E\left\{\sum_{x \in E}\exp{S_n\phi(x)}\right\}\right) = h(\mu) + \int\phi{d\mu},
\end{equation}
where the infimum is taken over all the $(\epsilon,n,\alpha)$-spanning subsets $E \subset M$.
\end{proposition}

The proof of Proposition \ref{main.technical.lemma} follows by fixing $\alpha > 0$, $\delta > 0$, $n > 0$ and a finite $(\delta,n,\alpha)$-spanning subset $E_0$ such that each 
$x \in E_0$ is endowed with a hyperbolic branch $f^{R(x)} : S_x \to U_x$ for a suitable return time to a hyperbolic Pesin set of 
generic points $\Lambda$.Then we choose a suitable subset of those hyperbolic branches to generate a horseshoe with finitely many 
branches and variable return times $\Omega^*$ and then we prove that $\Omega = \bigcup_{n \in \enteros}f^n(\Omega^*)$, the 
$f$-invariant saturate of $\Omega^*$ satisfies the estimatives (\ref{main.1}) in Proposition \ref{main.technical.lemma}.
\ 
\\
\\
Let $\{\psi_i\}$ be a {\bf countable dense subset of continuous functions}. 
\ 
\\
\\
{\bf Let $\rho > 0$ and $s > 0$ be fixed once for all}.
\ 
\\
\\
{\bf Choosing a hyperbolic Pesin set $\Lambda$ of quasi-generic points}

A crucial point in the construction is the choice of non invariant uniformly hyperbolic set $\Lambda$ also called Pesin set of quasi-generic points. 

It is not hard to convince that $\{\phi_n\}$ has tempered variation if and only if 
\begin{equation}\label{tempered.variation}
\limsup_{\delta \to 0^+}\limsup_{n \to +\infty}\dfrac{1}{n}\sup\{|\phi_n(x)-\phi_n(y)|: d(f^k(x),f^k(y)) < \delta, k = 0, \cdots , n-1\} = 0.
\end{equation}

\begin{lemma}\label{defining.Lambda}
There exists a Pesin set $\Lambda$ of generic points with $\mu(\Lambda) \geq 1-\rho$, an integer $N_0 > 0$ and a Borel subset 
$\Lambda_0 \subset \Lambda$ with $\mu(\Lambda_0) \geq (1-\rho)\mu(\Lambda)$ such that:
\begin{enumerate}
\item[a)] $\Phi | \Lambda$ is continuous;
\item[b)] in the subadditive case
\begin{equation}\label{uniformity.1}
\Phi(x) \leq \dfrac{\phi_m(x)}{m} < \Phi(x) + \rho \quad\forall \ x \in \Lambda \quad\forall \ m \geq N_0;  
\end{equation}
and, in the superadditive case,
\begin{equation}\label{uniformity.2}
\Phi(x) - \rho < \dfrac{\phi_m(x)}{m} \leq \Phi(x) \quad\forall \ x \in \Lambda \quad\forall \ m \geq N_0;  
\end{equation}
\item[c)]
\begin{equation}\label{quasigeneric.points}
\forall \ x \in \Lambda: \quad\left|\sum_{k=0}^{n-1}\psi_i(f^k(x)) - \int\psi_i{d\mu} \right| < \rho/2 \quad\forall i \leq s \quad\forall n \geq N_0
\end{equation}

\item[d)]
\begin{equation}\label{frequency.visits.Lambda}
\forall \ x \in \Lambda_0: \quad \dfrac{\#\{0 \leq j < n: f^j(x) \in \Lambda\}}{n} < 1+\rho \quad\forall \ n \geq N_0
\end{equation}
and
\begin{equation}\label{frecuency.visits.Lambda^c}
\forall \ x \in \Lambda_0: \quad \dfrac{\#\{0 \leq j < n: f^j(x) \in \Lambda^c\}}{n} < 2\rho \quad\forall \ n \geq N_0.
\end{equation}
\end{enumerate}
\end{lemma}

This follows from Egorov-Lusin theorem and the ergodicity of $\mu$. We refer to section \ref{proof.lemmas} for details.
\ 
\\
\\
{\bf Choosing $\alpha > 0$}
\ 
\\
\\
\emph{
We define $\alpha$ as
\begin{equation}\label{defining.alpha}
\alpha = \dfrac{\mu(\Lambda_0)}{2} 
\end{equation}
}
\ 
\\
\\
{\bf The definition of $\Phi_{\rho}$}

\begin{definition}
We define $\Phi_{\rho}$ to be a continuous extension of $\Phi \mid \Lambda$ with the sole condition that $\|\Phi_{\rho}\|_{\infty} \leq L$.
\end{definition}

Clearly $\Phi_{\rho} \to \Phi$ in measure as $\rho \to 0^+$. We shall see that $\Omega = \Omega(\rho,s,\Phi_{\rho})$ constructed previously 
for a continuous potential $\phi$ satisfies (\ref{main.1}).
\ 
\\
\\
{\bf Choosing a small precision $\delta > 0$ and $\mathcal{M}_0$}
\ 
\\
\begin{lemma}
There exists $\delta(\rho,s) > 0$ and $\mathcal{M}_0 = \{m_k\}_{k > 0}$ such that, for every $0 < \delta < \delta(\rho,s)$ it holds
\begin{equation}\label{definition.delta.1}
 \forall \ x,y \in M: \quad d(x,y) < \delta \Longrightarrow |\psi_{i}(x)-\psi_{i}(y)|  <  \rho/2, \quad\quad \forall \ i \leq s,
\end{equation}
\begin{equation}\label{definition.delta.2}
 \forall \ x,y \in M: \quad d(x,y) < \delta \Longrightarrow \left|\dfrac{\phi_m(x)}{m}-\dfrac{\phi_m(y)}{m}\right| < \rho, \quad\forall \ m \in \mathcal{M}_0.
\end{equation}
and
\begin{equation}\label{definition.delta.3}
\left|\lim_{n \to +\infty}\dfrac{1}{n}\log\left(\inf_E\left\{\sum_{x \in E}\exp{S_n\Phi_{\rho}}\right\}\right) - P_{\mu}(\Phi_{\rho})\right| < \rho/4.
\end{equation}
infimum is taken over all the $(\delta,n,\alpha)$-spanning subsets $E$. 
\end{lemma}
 
\begin{proof}
(\ref{definition.delta.1}) follows from the continuity of $\psi_i$; (\ref{definition.delta.2}) follows from the tempered variation condition 
(\ref{tempered.variation}) and (\ref{definition.delta.3}) follows from the definition of the limit (\ref{definition.free.energy.2}).
\end{proof}
\ 
\\
\\
{\bf Choosing a time $N_0 > 0$}
\ 
\\
Pesin set are endowed with covering by rectangles obtained from regular neighborhoods, that is, local coordinates at which $f$ looks like a small $C^1$ perturbation of a linear 
hyperbolic isomorphism. The diffeomorphism $f$ behaves as a uniformly hyperbolic map in these coordinates so preserving suitable continuous families of 
cones and therefore approximately local stable and unstable admissible manifolds so providing the structure of a hyperbolic branch similar to those 
used in the well-known construction of a horseshoe. The covering by these rectangles behaves under iterations of $f$ mostly as pieces of a 
a Markov partition. 

\begin{definition}
A finite covering of $\Lambda$ by rectangles ${\mathcal R} = \{\R_i\}$ is called a $(\delta,\kappa, \lambda)$-pseudo Markov covering if 
$\diam(R_i) < \delta$ for every $i$ and the following \emph{hyperbolic return property} holds true: 
there exists for every $\R_i$ a subrectangle $\subR_i \subset \R_i$ with $\diam(\subR_i) < \kappa$ such that
\begin{itemize}
\item $\Lambda \subset \bigcup_i\subR_i$;
\item for every $x \in \subR_i \cap \Lambda$ returning to $\subR_i \cap \Lambda$ after $m$-iterates there exists a hyperbolic branch
$$
f^m : S_i \to U_i
$$
where $S_i \subset \R_i$ (resp. $U_i \subset \R_i$) is an stable (resp. unstable) cylinder; moreover, the rate of nonlinear expansion along the 
unstable admissible manifolds is bounded from below by a constant $\lambda > 1$;
\item and
\begin{equation}\label{rho.shadowing.1}
\diam(f^j(S_x)) < \delta \quad\text{for every} \ j = 0, \cdots , R(x)-1.
\end{equation} 
\end{itemize}
\end{definition}

Now we fix a $(\delta/4,\kappa, \lambda)$-pseudo Markov covering of $\Lambda$.

\begin{lemma}\label{definition.n}
There exists $N_0 > 0$ such that, for every $n \geq N_0$ it holds 
\begin{equation}\label{definition.n.1}
\left|\dfrac{1}{n}\log\left(\inf_E\left\{\sum_{x \in E}\exp{S_n\Phi_{\rho}}\right\}\right) - P_{\mu}(\Phi_{\rho})\right| < \rho/2,
\end{equation}
\begin{equation}\label{definition.n.2}
\exp(n\rho) \geq \#{\mathcal R},
\end{equation}
\begin{equation}\label{definition.n.3}
\dfrac{L}{n(1+\rho)} < \rho. 
\end{equation}
and the whole set of conditions (\ref{uniformity.1}), (\ref{uniformity.2}), (\ref{quasigeneric.points}), (\ref{frequency.visits.Lambda}) and 
(\ref{frecuency.visits.Lambda^c}) still holds true.
\end{lemma}

\begin{lemma}\label{return.time.lemma}
There exists $N_0 > 0$ larger than the $N_0$ introduced at previous Lemma \ref{definition.n} such that for every 
subrectangle $\subR_i \subset \R_i$ of $(\delta/4,\kappa, \lambda)$-pseudo Markov covering of $\Lambda$ previouslky chosen 
there exists a subset $\Lambda_{0,i} \subset \subR_i \cap \Lambda_0$ with
$$
\mu(\Lambda_{0,i}) \geq \mu(\subR_i \cap \Lambda_0)/2  
$$
such that for every $x \in \Lambda_{0,i}$ there exists a return time $f^{R(x)}(x) \in \Lambda_{0,i}$ with
\begin{equation}\label{return.time}
R(x) \in [n,(1+\rho)n].
\end{equation}
\end{lemma}

This follows from the ergodicity of $\mu$: Let $A \subset M$ be a Borel set with $\mu(A) > 0$. Then given $\rho > 0$ and $n > 0$ define
$$
A_{\rho,n} := \{x \in A: x \quad\text{has a return time}\quad R(x) \in [n,(1+\rho)n]\}
$$
Then given $0 < \epsilon < 1$ there exists $N > 0$ and a Borel subset $A_{\epsilon} \subset A$ such that
$$
\mu(A_{\rho,n}) \geq (1-\epsilon)\mu(A) \quad\text{for every} \ n \geq N.
$$
Cf. \cite{katok.mendoza}. 
\ 
\\
\\
{\bf We fix once for all some $n \geq N_0$ satisfying conditions of Lemma \ref{definition.n} and Lemma \ref{return.time.lemma}}
\ 
\\
\\
{\bf Choosing $E_0$}
\ 
\\
\\
Now we notice that,
$$
\mu(\bigcup_{i}\Lambda_{0,i}) \geq \alpha.
$$
Therefore we can choose a maximal $(\delta,n)$ separated subset $E_0 \subset \bigcup_i\Lambda_{0,i}$ such that
\begin{equation}\label{main.estimative.1}
\left|\dfrac{1}{n}\log\left(\sum_{x \in E_0}\exp(S_n\Phi_{\rho}(x))\right) - P_{\mu}(\Phi_{\rho})\right| < \rho
\end{equation}
\ 
\\
\\
{\bf The construction of $\Omega$}
\ 
\\
\\
By construction for each point $x \in E_0$ there exists a hyperbolic branch $f^{R(x)}: S_x \to U_x$ contained in some $\R_i$ and such that
\begin{equation}\label{rho.shadowing.2}
\diam(f^j(S_x)) < \delta/4 \quad\text{for every} \ j = 0, \cdots , R(x)-1.
\end{equation}
This and the condition of separation of points in $E_0$ implies that any two different branches subordinated to the same rectangle are disjoint.

Moreover, by (\ref{definition.delta.1}), (\ref{rho.shadowing.2}) and (\ref{quasigeneric.points}) every such branch is 
\emph{$(\rho,s)$-generic}, that is,
\begin{equation}\label{rho-s.generic}
\left|\frac 1n\sum_{j=0}^{R(x)-1}\phi_{i}(f^{j}(x))-\int\phi_{i}d\mu\right|\leq \rho \quad \forall i\leq s
\end{equation}

Then we choose $\ell > 0$ and a subset 
$$
E_{\ell} := \subR_{\ell} \cap E_0
$$
such that
\begin{equation}\label{main.estimative.2}
\sum_{x \in E_{\ell}}\exp{S_n\Phi_{\rho}(x)} \geq \sum_{x \in E_{\ell'}}\exp{S_n\Phi_{\rho}(x)} \quad\text{for every}\quad \ell' \not= \ell,
\end{equation}
and define $\Omega(\rho,s,\Phi_{\rho})$ as the $f$-invariant saturate of the horseshoe with finitely many branches defined by the collection of branches 
$\{f^R(x): S_x \to U_x: x \in E_{\ell}\}$ chosen by condition (\ref{main.estimative.2}):
\begin{equation}\label{definition.Omega}
\Omega(\rho,s,\Phi_{\rho}) = \bigcup_{n \in \enteros}f^n\left(\bigcap_{n > 0}(f^{R})^{n}\bigcup_{x \in E_{\ell}}S_x\right), 
\end{equation}
where $f^R | S_x = f^{R(x)}$.

By \cite[Proposition 5.1]{luzzatto.sanchez}, all the ergodic $f$-invariant measures supported on $\Omega$ belongs to ${\mathcal O}(\rho,s)$ 
since the branches $\{f^{R(x)} : S_x \to U_x: x \in E_{\ell}\}$ are $(\rho,s)$-generic.

\section{Proof of Proposition \ref{main.technical.lemma}: estimating the pressure $P(f|\Omega,\phi_m/m)$, $m \in \mathcal{M}_0$}

To prove inequality (\ref{main.1}) in Main Technical Lemma we bound the topological pressure $P(f|\Omega,\phi)$ computed by the formula
\begin{equation}\label{topological.pressure.limit}
P(\Omega,\phi) = \limsup\limits_{N \to +\infty}\dfrac{1}{N}\log\left(\sum\limits_{x \in Per(N)}\exp{S_N\phi}\right).
\end{equation}
This was proved in \cite[Section 7.19 (7.11)]{ruelle.2004}.

\begin{lemma}\label{lemma.1}
For every $x \in E_0$ and for every $m \in \mathcal{M}_0$
\begin{equation}\label{bounded.variation.phi_m/m.1}
\left|S_{R(x)}\dfrac{\phi_m(z)}{m}-S_{R(x)}\dfrac{\phi_m(x)}{m}\right| < R(x)\rho \quad\forall \ z \in S_x.
\end{equation}
\end{lemma}
\begin{proof}
We use (\ref{definition.delta.2}) and $\diam(f^j(S_x)) < \delta/4$ for $j = 0, \cdots , R(x)-1$ to get (\ref{bounded.variation.phi_m/m.1}).
\end{proof}

We observe that for every $f$-periodic point $z \in \Omega$ there exists a unique finite subset 
$\{z_0, \cdots , z_{p-1}: p > 1\} \subset E^p_{\ell}$ which $\delta/4$-shadows the orbit of $z$ up to its return time, namely:
\begin{equation}\label{delta.shadowing.condition}
\dist(f^{j+\sum_{i < k}R(x_i)}(z), f^j(x_k)) < \delta/4 \quad\text{for}\quad j = 0, \cdots , R(x_k)-1 \quad\text{and every}\quad k = 0, \cdots p-1.
\end{equation}
where
\begin{equation}\label{admissible.periods.equation.0}
N = \sum\limits_{i=0}^{p-1}R(x_i), 
\end{equation}
is the period of $z$. Then we define
$$
\Delta(p) = \{N \in \naturales: \exists \ [x_0, \cdots , x_{p-1}] \in E^p_{\ell} \ \text{such that} \ N = \sum_iR(x_i)\}.
$$

\begin{lemma}\label{lemma.2}
For every $m \in \mathcal{M}_0$:
 \begin{eqnarray}
\label{main.estimative.3}
\sum_{N \in \Delta(p)}\sum\limits_{z \in Per(N)}\exp{S_N\left(\dfrac{\phi_m}{m}+\rho\right)(z)} \geq \left[\sum_{x \in E_{\ell}}\exp{S}_{R(x)}\left(\dfrac{\phi_m(x)}{m}\right)\right]^{p} \\
\label{main.estimative.4}
\sum_{N \in \Delta(p)}\sum\limits_{z \in Per(N)}\exp{S_N\left(\dfrac{\phi_m}{m}-\rho\right)(z)} \leq \left[\sum_{x \in E_{\ell}}\exp{S}_{R(x)}\left(\dfrac{\phi_m(x)}{m}\right)\right]^{p} 
\end{eqnarray}
\end{lemma}

This follows from (\ref{bounded.variation.phi_m/m.1}) by a shadowing argument. See next section for details.

\begin{lemma}\label{lemma.3}
For every $m \in \mathcal{M}_0$:
\begin{equation}\label{main.estimative.5}
\sum_{x \in E_{\ell}}\exp{S}_{R(x)}\left(\dfrac{\phi_m(x)}{m}\right) \geq \sum_{x \in E_{\ell}}\exp{S}_{n}\left(\dfrac{\phi_m(x)}{m}\right)\exp(-L).
\end{equation}
\begin{equation}\label{main.estimative.6}
\sum_{x \in E_{\ell}}\exp{S}_{R(x)}\left(\dfrac{\phi_m(x)}{m}\right) \leq \sum_{x \in E_{\ell}}\exp{S}_{n}\left(\dfrac{\phi_m(x)}{m}\right) \times \exp(n\rho{L}).
\end{equation}
\end{lemma}

\begin{lemma}\label{lemma.7}
\begin{equation}\label{main.estimative.9}
\forall \ m \in \mathcal{M}_0: \quad \sum_{x \in E_{\ell}}\exp{S}_{n}\left(\dfrac{\phi_m(x)}{m}\right) \geq \exp(n[P^*_{\mu}(\Phi)-o(1)])
\end{equation}
and
\begin{equation}\label{main.estimative.10}
\forall \ m \in \mathcal{M}_0: \quad \sum_{x \in E_{\ell}}\exp{S}_{n}\left(\dfrac{\phi_m(x)}{m}\right) < \exp(n[P^*_{\mu}(\Phi)+o(1)]).
\end{equation} 
\end{lemma}

Then, as an straightforward consequence of Lemma \ref{lemma.2}, Lemma \ref{lemma.3} and Lemma \ref{lemma.7} we get the 
\ 
\\
\\
{\bf Main estimatives}
$\forall \ m \in \mathcal{M}_0:$
$$
\sum_{N \in \Delta(p)}\sum\limits_{z \in Per(N)}\exp{S_N\left(\dfrac{\phi_m}{m}+\rho\right)(z)}  \geq \left[\exp(n[P^*_{\mu}(\Phi)-o(1)])\times\exp(-L)\right]^p.
$$
and
$$
\sum_{N \in \Delta(p)}\sum\limits_{z \in Per(N)}\exp{S_N\left(\dfrac{\phi_m}{m}-\rho\right)(z)} \leq \left[\exp(n[P^*_{\mu}(\Phi)+o(1)])\times \exp(n\rho{L})\right]^p. 
$$

\begin{proof}{[Proof of Proposition \ref{main.technical.lemma} (\ref{main.1})] using the \bf{main estimatives}}
\ 
\\
We start noting that
\begin{equation}\label{cardinal.admissible.periods}
1 \leq \#\Delta(p) \leq np\rho,
\end{equation}
since
\begin{equation}\label{admissible.periods.bounds}
np \leq N \leq n(1+\rho)p \quad\text{for every}\quad N \in \Delta(p). 
\end{equation}
Moreover,
\begin{equation}\label{counting.cycles}
\dfrac{N}{n(1+\rho)} \leq p \leq \dfrac{N}{n},
\end{equation}
since $R(x_k) \in [n,(1+\rho)n]$ for every $k = 0, \cdots , p-1$. 

Now we use our {\bf main estimatives}. Let $m \in \mathcal{M}_0$.

We first choose a period $N^+_p \in \Delta(p)$ where $N \longmapsto \sum_{z \in Per(N)}\exp(S_N(\phi+\rho)(z))$ attains its maximum 
over the set of admissible periods $N \in \Delta(p)$. Then we get
\begin{eqnarray*}
\#\Delta(p)\sum\limits_{z \in Per(N^+_p)}\exp{S}_{N^+_p}\left(\dfrac{\phi_m}{m}+\rho\right)(z) & \geq & \sum_{N \in \Delta(p)}\sum\limits_{z \in Per(N)}\exp{S_N\left(\dfrac{\phi_m}{m}+\rho\right)(z)}\\
                                                                                               & \geq & \left[\exp(n[P^*_{\mu}(\Phi)-o(1)])\exp(-L)\right]^{p}\\  
                                                                                               & \geq & \left[\exp(n[P^*_{\mu}(\Phi)-o(1)])\exp(-L)\right]^{\frac{N^{+}_p}{(1+\rho)n}},
\end{eqnarray*}
using inequality (\ref{counting.cycles}) to bound from below $p > 0$ in terms of $N^{+}_p$. Then, by (\ref{cardinal.admissible.periods}),
\begin{eqnarray*}
np\rho\times\sum\limits_{z \in Per(N^{+}_p)}\exp{S_{N^{+}_p}\left(\dfrac{\phi_m}{m}+\rho)\right)(z)} \geq  \\
\left[\exp(n[P^*_{\mu}(\Phi)-o(1)])\times\exp(-L)\right]^{\frac{N^{+}_p}{(1+\rho)n}} 
\end{eqnarray*}
Similarly,  minimizing the sums $\sum_{z \in Per(N)}\exp(S_N(\phi-\rho)(z))$ over the set of admissible periods $N \in \Delta(p)$, using again 
(\ref{cardinal.admissible.periods}),we find an admissible period $N^-_p \in \Delta(p)$ with $N^{-}_p \in [np,n(1+\rho)p]$ such that
\begin{eqnarray*}
\sum\limits_{z \in Per(N^{-}_p)}\exp{S_{N^{-}_p}\left(\dfrac{\phi_m}{m}-\rho\right)(z)} \leq \#\Delta(p)\sum\limits_{z \in Per(N^{-}_p)}\exp{S_{N^{-}_p}\left(\dfrac{\phi_m}{m}-\rho\right)(z)}\\
\leq \sum_{N \in \Delta(p)}\sum\limits_{z \in Per(N)}\exp{S_N\left(\dfrac{\phi_m}{m}-\rho\right)(z)} \leq \left[\exp(n[P^*_{\mu}(\Phi)+o(1)])\times \exp(n\rho{L})\right]^{p}\\
\leq \left[\exp(n[P^*_{\mu}(\Phi)+o(1)])\times \exp(n\rho{L})\right]^{\frac{N^{-}_p}{n}}
\end{eqnarray*}
Then, taking logarithms, dividing by $N^{+}_p$ (resp. $N^{+}_p$ ) and passing to the limit as $p \to +\infty$, we get that
$$
P\left(f|\Omega,\dfrac{\phi_m}{m}+\rho\right) \geq \dfrac{P^*_{\mu}(\Phi)-o(1)}{1+\rho} - \dfrac{L}{n(1+\rho)}
$$
and
$$
P\left(f|\Omega,\dfrac{\phi_m}{m}-\rho\right) \leq P^*_{\mu}(\Phi)+o(1) + \rho{L},
$$
for every $m \geq M_0$.
Therefore, as $L/n(1+\rho) < \rho$ by (\ref{definition.n.3}), we have after a straightforward calculation, using $P(\phi + c) = P(\phi)$ ([Theorem 2.1, (vii)]\cite{walters}),
$$
P\left(f|\Omega,\dfrac{\phi_m}{m}\right) \geq \dfrac{P^*_{\mu}(\Phi)-o(1)}{1+\rho}
$$
and
$$
P\left(f|\Omega,\dfrac{\phi_m}{m}\right) \leq P^*_{\mu}(\Phi) + o(1),
$$
for every $m \in \mathcal{M}_0$, so proving (\ref{main.1}) in Proposition \ref{main.technical.lemma}.
\end{proof}

\section{Proofs of the lemmas}\label{proof.lemmas}

\begin{proof}{[Proof of Lemma \ref{defining.Lambda}]}
\ 
\\
Given $\rho > 0$ and $N > 0$ we define, for a subadditive sequence $\phi_n$,
\begin{equation}
X_{N} := \{x \in M: \Phi(x) \leq \dfrac{\phi_m(x)}{m} \leq \Phi(x) + \rho, \quad \forall \ m \geq N\}.
\end{equation}
Moreover, each $X_{N}$ is compact: given a sequence $x_n \in X_{N}$, then 
for every $m \geq N$:
$$
\Phi(x) \leq \dfrac{\phi_m(x)}{m} = \limsup_{n \to +\infty}\dfrac{\phi_m(x_n)}{m} \leq \limsup_{n \to +\infty}\Phi(x_n) + \rho \leq \Phi(x) + \rho,
$$
since $\Phi$ is uppersemicontinuous, therefore $x \in X_{N}$, so each $X_{N}$ is closed, hence compact in $M$.

If $\phi_n$ is superadditive we define
\begin{equation}
X_{N} := \{x \in M: \Phi(x) - \rho \leq \dfrac{\phi_m(x)}{m} \leq \Phi(x), \quad \forall \ m \geq N\}.
\end{equation} 
and prove that it is compact using that $\Phi$ is lowersemicontinuous.

Clearly $X_{N} \subset X_{N+1}$ and $M = \bigcup_{N > 0}X_{N}$ in both cases sub(super)additive and
$$
\mu(X_{N}) \uparrow 1 \quad\text{as}\quad N \to +\infty.
$$

Let $X_{\Phi}$ be a compact subset such that $\Phi | X_{\Phi}$ is continuous with large measure which exists by the Egorov-Lusin theorem. 

Now we find a compact Pesin set $\Sigma$ such that
$$
\Sigma_N := \{x \in \tilde\Lambda: \left|\sum_{k=0}^{n-1}\psi_i(f^k(x)) - \int\psi_i{d\mu} \right| < \rho/2 \quad\forall i \leq s \quad\forall n \geq N\}
$$
has
$$
\mu(\Sigma_N) \to \mu(\Sigma) \quad\text{as}\quad N \to +\infty.
$$ 
Then we choose $X_{\Phi}$ and $\Sigma$ by the Egorov-Lusin theorem such that there exists $N(\Lambda) > 0$ where
$$
\Lambda := \Sigma_N \cap X_{N} \cap X_{\Phi} \quad\text{has}\quad \mu(\Lambda) \geq 1-\rho \quad\forall \ N \geq N(\Lambda).
$$
Let $\Lambda_0 \subset \Lambda$ be a Borel subset with $\mu(\Lambda_0) \geq (1-\rho)\mu(\Lambda)$ 
and $N(frecuency) > 0$ a large integer such that (\ref{frequency.visits.Lambda}) and (\ref{frecuency.visits.Lambda^c}) holds true 
for every $n \geq N(frecuency)$.
This is possible by the ergodicity of $\mu$. Actually,
$$
1 - \rho \leq \lim_{n \to +\infty}\dfrac{\#\{0 \leq j < n: f^j(x) \in \Lambda\}}{n} = \mu(\Lambda) \leq 1 \quad\mu-a.e. \ x \in M.
$$
Then define
$$
Y_N = \left\{x \in M: 1 - 2\rho \leq \dfrac{\#\{0 \leq j < n: f^j(x) \in \Lambda\}}{n} \leq 1 + \rho, \quad\forall \ n \geq N\right\}.
$$
Notice that if $x \in Y_N$ then
$$
\dfrac{\#\{0 \leq j < n: f^j(x) \in \Lambda^c_1\}}{n} \leq 2\rho, \quad\forall \ n \geq N.
$$
As $Y_N \subset Y_{N+1}$ and $M = \bigcup_NY_N$ then $\mu(Y_N) \uparrow 1$ when $N \to +\infty$, therefore we can choose 
$N(frecuency) > 0$ sufficiently large such that 
$$
\mu(\Lambda \cap Y_N) \geq (1-\rho)\mu(\Lambda) \quad\text{for every}\quad N \geq N(frecuency).
$$

We thus define
$$
N_0 := \max\{N(\Lambda), N(frecuency)\}
$$
This completes the proof.
\end{proof}
\begin{remark}
Notice that we can choose any larger $N_0$ with the same set of conditions (\ref{uniformity.1}), 
(\ref{uniformity.2}), (\ref{frequency.visits.Lambda}) and (\ref{frecuency.visits.Lambda^c}) in Lemma \ref{defining.Lambda}.
\end{remark}
\ 
\\
\begin{proof}{[Proof of Lemma \ref{lemma.2}]}
\ 
\\
Let $z \in Per(N)$, $N \in \Delta(p)$ and, by (\ref{rho.shadowing.2}), $[x_0, \cdots , x_{p-1}] \in E^p_{\ell}$ an ordered sequence of points in 
$E_{\ell}$ which successively $\delta$-shadows the orbit of $z$ up to its return time as in (\ref{delta.shadowing.condition}). Then by our 
choice of $\delta$ and $\mathcal{M}_0$ in (\ref{definition.delta.2}), for every $m \in \mathcal{M}_0$,
$$
\left|\dfrac{\phi_m(f^{j+\sum_{k < i}R(x_k)}(z))}{m} - \dfrac{\phi_m(f^j(x_i))}{m} \right| < \rho, \quad\forall \ j = 0, \cdots, R(x_i)-1, \quad\forall \ i = 0, \cdots, p-1,
$$
and we thus get
\begin{eqnarray*}
\left|\sum_{j=0}^{N-1}\dfrac{\phi_m(f^j(z))}{m} - \sum_{i=0}^{p-1}\sum_{j=0}^{R(x_i)-1}\dfrac{\phi_m(f^j(x_i))}{m} \right| \\
= \left|\sum_{i=0}^{p-1}\sum_{j=0}^{R(x_i)-1}\dfrac{\phi_m(f^{j+\sum_{k < i}R(x_k)}(z))}{m} - \sum_{i=0}^{p-1}\sum_{j=0}^{R(x_i)-1}\dfrac{\phi_m(f^j(x_i))}{m} \right| \\
< \sum_{i=0}^{p-1}\sum_{j=0}^{R(x_i)-1}\rho = N\rho, 
\end{eqnarray*}
since $N = \sum_{i=0}^{p-1}R(z_i) \in \Delta(p)$. Therefore, for every $m \in \mathcal{M}_0$,
\begin{eqnarray*}
\sum_{N \in \Delta(p)}\sum\limits_{z \in Per(N)}\exp{S_N}\left(\dfrac{\phi_m}{m}+\rho\right)(z) & \geq & \sum_{[x_0, \cdots , x_{p-1}] \in E^p_{\ell}}\sum_{j=0}^{R(x_i)-1}\dfrac{\phi_m(f^j(x_i))}{m}\\
                                                                                                             &  =   & \left[\sum_{x \in E_{\ell}}\exp{S}_{R(x)}\left(\dfrac{\phi_m(x)}{m}\right)\right]^{p}.
\end{eqnarray*}
Similarly for (\ref{main.estimative.4}).
\end{proof}
\ 
\\
\begin{proof}[Proof of Lemma \ref{lemma.3}]
\ 
\\
As $R(x) \in [n,(1+\rho)n]$ for every $x \in E_0$ then $\min_{x \in E_{\ell}}R(x)-n \geq 0$. So estimative 
(\ref{main.estimative.5}) follows since then 
\begin{eqnarray*}
S_{R(x)}\dfrac{\phi_m}{m}(x) = S_n\dfrac{\phi_m}{m}(x) + \sum_{j = n}^{R(x)-1}\dfrac{\phi_m}{m}(x) \\
\geq S_n\dfrac{\phi_m}{m}(x) + (\min_{x \in E_{\ell}}(R(x)-1-n)\inf\dfrac{\phi_m}{m}\\
\geq S_n\dfrac{\phi_m}{m}(x) - \inf\dfrac{\phi_m}{m}\\
\geq S_n\dfrac{\phi_m}{m}(x) - L
\end{eqnarray*}
by (\ref{bounds.subadditive.potential.1}). In the same manner we prove (\ref{main.estimative.6}) using that $\max_{x \in E_{\ell}}R(x)-n \leq n\rho$ 
since then
\begin{eqnarray*}
S_{R(x)}\dfrac{\phi_m}{m}(x) \leq S_{n}\dfrac{\phi_m}{m}(x) + (\max_{x \in E_{\ell}}R(x)-n)\sup\dfrac{\phi_m}{m} \\
\leq S_{n}\dfrac{\phi_m}{m}(x) + n\rho\sup\dfrac{\phi_m}{m}\\
\leq S_{n}\dfrac{\phi_m}{m}(x) + n\rho{L}
\end{eqnarray*}
\end{proof}

\begin{lemma}\label{lemma.4}
For every $x \in \Lambda_0$, $n \geq N_0$ and $m \in \mathcal{M}_0$,
\begin{equation}\label{variation.phi_m/m-Phi_rho}
\left|S_n\dfrac{\phi_m(x)}{m} - S_n\Phi_{\rho}(x)\right| < no(1)
\end{equation}
\end{lemma}
\begin{proof}
Using Lemma \ref{defining.Lambda}, we have that for every $x \in \Lambda_0$,
\begin{eqnarray*}
\left|S_n\dfrac{\phi_m(x)}{m}-S_n\Phi_{\rho}(x)\right| & \leq & \sum_{k=0}^{n-1}\left|\dfrac{\phi_m(f^k(x))}{m}-\Phi_{\rho}(f^k(x))\right|\\
                                                       & \leq & \rho\#\{0 \leq j < n: f^k(x) \in \Lambda\} + 2L\#\{0 \leq j < n: f^k(x) \in \Lambda^c\} \\
                                                              & \leq &  n\rho(1+\rho) + 4n\rho{L}\\
                                                              & \leq & 2n\rho + 4n\rho{L} =  no(1), \quad\forall \ m \in \mathcal{M}_0, \ \forall \ n \geq N_0.
\end{eqnarray*}
\end{proof}

\begin{lemma}\label{lemma.5}
\begin{equation}\label{main.estimative.7}
\forall \ m \in \mathcal{M}_0: \quad \left|\dfrac{1}{n}\log\left(\sum_{x \in E_0}\exp{S_n\dfrac{\phi_m(x)}{m}}\right) - P^*_{\mu}(\Phi)\right| < o(1).
\end{equation} 
\end{lemma}
\begin{proof}
Let $m \in \mathcal{M}_0$. Adding and subtracting terms in the left side of the inequality (\ref{main.estimative.7}) we have:
\begin{eqnarray*}
\left|\dfrac{1}{n}\log\left(\sum_{x \in E_0}\exp{S_n\dfrac{\phi_m(x)}{m}}\right) - P^*_{\mu}(\Phi)\right| & \leq & \left|\dfrac{1}{n}\log\left(\sum_{x \in E_0}\exp{S_n\Phi_{\rho}}\right) - P_{\mu}(\Phi_{\rho})\right|  + \\
                                                                                                                                &      & + \left|\dfrac{1}{n}\log\sum_{x \in E_0}\exp\left(S_n\dfrac{\phi_m(x)}{m}\right)-\dfrac{1}{n}\log\sum_{x \in E_0}\exp\left(S_n\Phi_{\rho}(x)\right)\right|\\                                                                                                                                
                                                                                                                                &      & + \left|P_{\mu}(\Phi_{\rho}) - P^*_{\mu}(\Phi)\right|.
\end{eqnarray*}
Notice that
$$
\left|P_{\mu}(\Phi_{\rho}) - P^*_{\mu}(\Phi)\right| = \left|\int\Phi_{\rho}d\mu - \int\Phi{d\mu}\right| \leq 2L\mu(\Lambda^c) < 2L\rho = o(1).
$$
and that
$$
\left|\dfrac{1}{n}\log\left(\sum_{x \in E_0}\exp{S_n\Phi_{\rho}}\right) - P_{\mu}(\Phi_{\rho})\right| < \rho
$$
by the choice of $n$ and $E_0$.
As for the second line in the inequality, we observe that
$$
\dfrac{\sum_{x \in E_0}\exp\left(S_n\dfrac{\phi_m(x)}{m}\right)}{\sum_{x \in E_0}\exp\left(S_n\Phi_{\rho}(x)\right)} \leq \max_{x \in E_0}\dfrac{\exp\left(S_n\dfrac{\phi_m(x)}{m}\right)}{\exp\left(S_n\Phi_{\rho}(x)\right)} \leq e^{no(1)}
$$
by (\ref{variation.phi_m/m-Phi_rho}) in Lemma \ref{lemma.4}.

Therefore, for every $m \in \mathcal{M}_0$ the three terms in the right side are all less than $o(1)$ so proving (\ref{main.estimative.7}).
\end{proof}

\begin{lemma}\label{lemma.6}
For every $m \in \mathcal{M}_0$
\begin{equation}\label{main.estimative.8}
\exp(o(1)n)\sum_{x \in E_{\ell}}\exp{S_n\dfrac{\phi_m(x)}{m}} \geq \sum_{x \in E_{\ell'}}\exp{S_n\dfrac{\phi_m(x)}{m}}, \quad \forall \ \ell' \not= \ell.
\end{equation}
\end{lemma}
\begin{proof}
By (\ref{variation.phi_m/m-Phi_rho}) in Lemma \ref{lemma.4},
$$
S_n\dfrac{\phi_m(x)}{m}-S_n\Phi_{\rho}(x) \geq - o(1)n \quad\forall \ m \in \mathcal{M}_0, \quad\forall \ n \geq N_0,
$$
for every $x \in \Lambda_0$. Hence,
\begin{eqnarray*}
\sum_{x \in E_{\ell}}\exp{S_n\dfrac{\phi_m(x)}{m}} & \geq & \exp(-o(1)n)\sum_{x \in E_{\ell}}\exp{S_n\Phi_{\rho}(x)}\\
                                                   & \geq & \exp(-o(1)n)\sum_{x \in E_{\ell'}}\exp{S_n\Phi_{\rho}(x)}, \quad\forall \ \ell' \not= \ell.                                                   
\end{eqnarray*}
by (\ref{main.estimative.2}) and then, once again, using (\ref{variation.phi_m/m-Phi_rho}), we get
$$
\sum_{x \in E_{\ell'}}\exp{S_n\Phi_{\rho}(x)} \geq \exp(-o(1)n)\sum_{x \in E_{\ell'}}\exp{S_n\dfrac{\phi_m(x)}{m}},
$$
so proving (\ref{main.estimative.8}).
\end{proof}

\begin{proof}[Proof of Lemma \ref{lemma.7}]
\ 
\\
\\
This follows from (\ref{main.estimative.7}) and (\ref{main.estimative.8}) and (\ref{definition.n.2}) in our choice of $n$. Indeed, let $m \in \mathcal{M}_0$, then
\begin{eqnarray*}
\exp(n\rho)\exp(o(1)n)\sum_{x \in E_{\ell}}\exp{S_n\dfrac{\phi_m(x)}{m}}   & \geq & \#{\mathcal R}\exp(o(1)n)\sum_{x \in E_{\ell}}\exp{S_n\dfrac{\phi_m(x)}{m}}\\
\sum_{\ell'}\sum_{x \in E_{\ell'}}\exp{S_n\dfrac{\phi_m(x)}{m}} & \geq & \sum_{x \in E_0}\exp{S_n\dfrac{\phi_m(x)}{m}}\\
                                                                                              & \geq &  \exp(n[P^*_{\mu}(\Phi) - o(1)]) 
\end{eqnarray*}
and thus
\begin{eqnarray*}
\sum_{x \in E_{\ell}}\exp{S_n\dfrac{\phi_m(x)}{m}}  & \geq & \exp(n[P^*_{\mu}(\Phi) - o(1)])\exp(-n\rho)\exp(-o(1)n)\\
                                                    & = & \exp(n[P^*_{\mu}(\Phi) - o(1)]).
\end{eqnarray*}
By the other side,
$$
\sum_{x \in E_{\ell}}\exp{S}_{n}\left(\dfrac{\phi_m(x)}{m}\right) \leq \sum_{x \in E_0}\exp{S}_{n}\left(\dfrac{\phi_m(x)}{m}\right) \leq \exp(n[P^*_{\mu}(\Phi)+o(1)]).                                                                
$$
\end{proof}

\section{Proof of Proposition \ref{limit.pressure.lemma}}\label{proof.limit.pressure.lemma}

Let $f : M \to M$ be a continuous transformation of a compact metric space $(M,d)$ and $\{\phi_n\}$ be a subadditive sequence of continuous functions. We shall suppose that there exists $L > 0$ such that
\begin{equation}\label{bound.rate.growing}
\left|\dfrac{\phi_n(x)}{n}\right| \leq L \quad\mu-a.e. \quad \forall \ \mu \in {\mathcal M}_f.
\end{equation}
Then, by Kingman's theorem there exists an uppersemicontinuous function $\Phi$ such that
\begin{equation}\label{definition.Phi.1}
\Phi = \inf_{n > 0}\dfrac{\phi_n}{n} \quad\mu-a.e. \quad \forall \ \mu \in {\mathcal M}_f
\end{equation} 
In particular, $|\Phi(x)| \leq L$ for $\forall\mu-a.e.$ and $\forall\mu \in {\mathcal M}_f$.

\begin{lemma}\label{lemma.1.1}
Let $\{\phi_n\}$ be a sub(super)additive sequence of continuous functions. Then,
\begin{equation}\label{pressure.sub}
\lim\limits_{n \to +\infty}P\left(\dfrac{\phi_n}{n}\right) = \inf\limits_{n > 0}P\left(\dfrac{\phi_n}{n}\right)
\end{equation}
and
\begin{equation}\label{pressure.sup}
\lim\limits_{n \to +\infty}P\left(\dfrac{\phi_n}{n}\right) = \sup\limits_{n > 0}P\left(\dfrac{\phi_n}{n}\right)
\end{equation}
\end{lemma}

This proves that, for a superadditive sequence $\{\phi_n\}$ we have
$$
P^*(\Phi) = \sup\limits_{n > 0}P\left(\dfrac{\phi_n}{n}\right).
$$
This follows from Kingman's theorem since
$$
\Phi = \sup_{n > 0}\dfrac{\phi_n}{n}
$$
and therefore
\begin{eqnarray*}
P^*(\Phi) & = & \sup\limits_{\mu \in {\mathcal M}_f}\left\{h(\mu) + \int\Phi{d\mu} \right\}\\
          & = & \sup\limits_{\mu \in {\mathcal M}_f}\left\{h(\mu) + \sup_{n > 0}\int\dfrac{\phi_n}{n}{d\mu} \right\}\\
          & = & \sup\limits_{n > 0}P\left(\dfrac{\phi_n}{n}\right).
\end{eqnarray*}

Hence, to complete the proof of Proposition \ref{limit.pressure.lemma} we are led to prove that, for subadditive sequences we have
\begin{equation}\label{pressure.sub.0}
P^*(\Phi)  =  \inf\limits_{n > 0}P\left(\dfrac{\phi_n}{n}\right).
\end{equation}

For this we first observe that
\begin{equation}\label{pressure.sub.1}
P^*(\Phi)  \leq  \inf\limits_{n > 0}P\left(\dfrac{\phi_n}{n}\right).
\end{equation}
This holds since $P^*(\Phi) \leq P(\phi_n/n)$ for every $n > 0$, using that $\Phi \leq \phi_n/n$ and the additive variational principle.

Therefore, we are going to prove that
\begin{equation}\label{pressure.sub.2}
\forall \ \epsilon > 0: \quad \inf\limits_{n > 0}P\left(\dfrac{\phi_n}{n}\right) < P^*(\Phi) + \epsilon.
\end{equation}

For this we introduce the sets
\begin{equation}\label{definition.M_N}
{\mathcal M}_N = \left\{\nu \in {\mathcal M}_f: \int\dfrac{\phi_n}{n}d\nu < \int\Phi{d\nu} + \epsilon \ \forall n \geq N \right\}
\end{equation}
and denote
\begin{equation}\label{definition.P_N}
P_N(\phi) = \sup\limits_{\nu \in {\mathcal M}_N}\left\{h(\nu) +  \int\phi{d\nu} \right\}.
\end{equation}
Notice that ${\mathcal M}_{N} \subset {\mathcal M}_{N+1}$ and that, by Kingman's theorem
$$
\bigcup_N{\mathcal M}_N = {\mathcal M}_f.
$$
Indeed, as
$$
\int\Phi{d\nu} = \inf\limits_{n > 0}\int\dfrac{\phi_n}{n}d\nu, \quad\forall \ \nu \in {\mathcal M}_f, 
$$
then, for every $\nu \in {\mathcal M}_f$ there exists $N = N(\epsilon,\nu)$ such that
$$
\int\dfrac{\phi_n}{n}d\nu < \int\Phi{d\nu} + \epsilon, \quad\forall n \geq N, 
$$
that is, $\nu \in {\mathcal M}_{N}$. By the definition of ${\mathcal M}_N$, we have that, given $N > 0$ then for every $\nu \in {\mathcal M}_N$,
$$
h(\nu) + \int\dfrac{\phi_n}{n}d\nu < P^*(\Phi) + \epsilon, \quad\forall n \geq N,
$$
and then $P_N\left(\dfrac{\phi_n}{n}\right) \leq P^*(\Phi) + \epsilon$, for every $n \geq N$ so that
$$
\forall \ N > 0: \quad \inf\limits_{n > 0}P_N\left(\dfrac{\phi_n}{n}\right) \leq P^*(\Phi) + \epsilon.
$$ 
We thus have that, for every $\epsilon > 0$,
\begin{equation}\label{pressure.sub.3}
\sup_{N > 0}\inf\limits_{n > 0}P_N\left(\dfrac{\phi_n}{n}\right) \leq P^*(\Phi) + \epsilon.
\end{equation}
Therefore, to prove (\ref{pressure.sub.1}) it is sufficient to show that
\begin{equation}\label{pressure.sub.4}
\sup\limits_{N > 0}\inf\limits_{n > 0}P_N\left(\dfrac{\phi_n}{n}\right) = \inf\limits_{n > 0}P\left(\dfrac{\phi_n}{n}\right).
\end{equation}

\begin{lemma}\label{lemma.1.2}
For every $N > 0$
\begin{equation}\label{pressure.sub.5}
\lim\limits_{n \to +\infty}P_{N}\left(\dfrac{\phi_n}{n}\right) = \inf\limits_{n > 0}P_{N}\left(\dfrac{\phi_n}{n}\right)
\end{equation}
\end{lemma}

\begin{lemma}\label{lemma.1.3}
\begin{equation}\label{pressure.sub.6}
P\left(\dfrac{\phi_n}{n}\right) = \lim_{N \to +\infty}P_N\left(\dfrac{\phi_n}{n}\right) = \sup_{N > 0}P_N\left(\dfrac{\phi_n}{n}\right).
\end{equation}
Moreover,
\begin{equation}\label{pressure.sub.6.1}
P^*(\Phi) = \lim_{N \to +\infty}P^*_N(\Phi) = \sup_{N > 0}P^*_N(\Phi).
\end{equation}
\end{lemma}

Therefore, (\ref{pressure.sub.4}) will be established by interchanging the order of the limits. This follows from

\begin{lemma}\label{lemma.1.4}
$\{P_N(\phi_n/n)\}$ is a Cauchy sequence in $(n,N)$. 
\end{lemma}

\begin{proof}{\em[Proof of Proposition \ref{limit.pressure.lemma}]}
We first observe that (\ref{pressure.sub.1}) and (\ref{pressure.sub.2}) proves (\ref{pressure.sub.0}). To prove (\ref{pressure.sub.2}) 
we observe that, by Lemma \ref{lemma.1.4} the limits in (\ref{pressure.sub.5}) and (\ref{pressure.sub.6}) can be interchanged, that is, 
\begin{equation}
\begin{array}{ccc}
\sup\limits_{N > 0}\inf\limits_{n > 0}P_N\left(\dfrac{\phi_n}{n}\right) & = & \lim\limits_{N \to +\infty}\lim\limits_{n \to +\infty}P_N\left(\dfrac{\phi_n}{n}\right)\\
= \lim\limits_{n \to +\infty}\lim\limits_{N \to +\infty}P_N\left(\dfrac{\phi_n}{n}\right) & = & \inf\limits_{n > 0}P\left(\dfrac{\phi_n}{n}\right).
\end{array}
\end{equation}
We use this and (\ref{pressure.sub.3}) to prove (\ref{pressure.sub.2}).
\end{proof}

\subsection{Proof of the lemmas}

\begin{proof}[Proof of Lemma \ref{lemma.1.1}]
 Let $\{\phi_n\}$ be a subadditive continuous potential. Using subadditivity and letting $m = np + q$ for some $0 \leq q < n$, we have
$$
\phi_m \leq \phi_{np}\circ f^q + \phi_q \leq \sum_{k=0}^{p-1}\phi_n \circ f^{q + kn} + qL
$$
since $\|\phi_n\| \leq nL$ by assumption. Therefore, for every $f$-invariant Borel probability $\mu \in {\mathcal M}_f$
\begin{eqnarray*}
\int\dfrac{\phi_m}{m}d\mu & \leq &  \dfrac{1}{m}\int\sum_{k=0}^{p-1}\phi_n \circ f^{q + kn}d\mu + \dfrac{qL}{m} \\
                                         & \leq &  \dfrac{np}{m}\int\dfrac{\phi_n}{n}d\mu + \dfrac{nL}{m}\\
                                         & \leq & \int\dfrac{\phi_n}{n}d\mu + \epsilon \quad\forall \ m \geq \dfrac{nL}{\epsilon}.
\end{eqnarray*}
Therefore,
\begin{equation}\label{1.1.1}
\int\dfrac{\phi_m}{m}d\mu  \leq \int\dfrac{\phi_n}{n}d\mu + \epsilon \quad\forall \ m \geq \dfrac{nL}{\epsilon}, \quad\forall \ \nu \in {\mathcal M}_f, 
\end{equation}
Then, substituting into (\ref{additive.variational.principle}), the variational equation for topological pressure, we get
$$
P\left(\dfrac{\phi_m}{m}\right) \leq P\left(\dfrac{\phi_n}{n}\right) + \epsilon \quad\text{for every}\quad m \geq \dfrac{nL}{\epsilon}. 
$$
Therefore,
$$
\limsup_{m \to +\infty}P\left(\dfrac{\phi_m}{m}\right) \leq P\left(\dfrac{\phi_n}{n}\right), \quad\text{for every}\quad n > 0.
$$
since $\epsilon > 0$ is arbitrary. Hence,
$$
\limsup_{m \to +\infty}P\left(\dfrac{\phi_m}{m}\right) \leq  \inf\limits_{n > 0}P\left(\dfrac{\phi_n}{n}\right)
                                                                \leq  \liminf_{m \to +\infty}P\left(\dfrac{\phi_m}{m}\right)
$$
concluding that the limit exists and
$$
\lim_{n \to +\infty}P\left(\dfrac{\phi_m}{m}\right) = \inf\limits_{n > 0}P\left(\dfrac{\phi_n}{n}\right).
$$

The superadditive case follows from similar arguments. We start remarking that
$$
\phi_m \geq \phi_{np}\circ f^q + \phi_q \geq \sum_{k=0}^{p-1}\phi_n \circ f^{q + kn} - qL,
$$
and then that
$$
\int\dfrac{\phi_m}{m}d\mu \geq \int\dfrac{np}{m}\dfrac{\phi_n}{n}d\mu -\dfrac{nL}{m} \quad\forall \ \mu \in {\mathcal M}_f.
$$
This proves
$$
P\left(\dfrac{\phi_m}{m}\right)  \geq  P\left(\dfrac{np}{m}\dfrac{\phi_n}{n}\right) - \dfrac{nL}{m}
$$
and then, 
\begin{eqnarray*}
P\left(\dfrac{\phi_m}{m}\right) - P\left(\dfrac{\phi_n}{n}\right) & \geq & P\left(\dfrac{np}{m}\dfrac{\phi_n}{n}\right)- P\left(\dfrac{\phi_n}{n}\right) - \dfrac{nL}{m}\\
                                                                  & \geq & -\left(\dfrac{np}{m}-1\right)\left\|\dfrac{\phi_n}{n}\right\|_{\infty} - \dfrac{nL}{m}\\
                                                                  & > & -\epsilon
\end{eqnarray*}
for every $m \geq M$, for a suitable $M$.
Therefore,
$$
\liminf_{m \to +\infty}P\left(\dfrac{\phi_m}{m}\right) \geq \sup_{n > 0}P\left(\dfrac{\phi_n}{n}\right) \geq \limsup_{n \to +\infty}P\left(\dfrac{\phi_n}{n}\right),
$$
proving that that the limit exists and is
$$
\lim_{n \to +\infty}P\left(\dfrac{\phi_n}{n}\right) = \sup_{n > 0}P\left(\dfrac{\phi_n}{n}\right).
$$
\end{proof}

\begin{proof}[Proof of Lemma \ref{lemma.1.2}]
Is the same as the proof of Lemma \ref{lemma.1.1}.
\end{proof}

\begin{proof}[Proof of Lemma \ref{lemma.1.3}]
As ${\mathcal M}_N \subset {\mathcal M}_{N+1}$ then $P_{N}(\phi_n/n) \leq P_{N+1}(\phi_n/n)$. Therefore, 
\begin{equation}\label{3.2}
\lim_{N \to +\infty}P_{N}\left(\dfrac{\phi_n}{n}\right) = \sup_{N > 0}P_{N}\left(\dfrac{\phi_n}{n}\right) \leq P\left(\dfrac{\phi_n}{n}\right).
\end{equation}
Given $\epsilon > 0$ there exists $\nu \in {\mathcal M}_f$ such that
$$
P\left(\dfrac{\phi_n}{n}\right) - \epsilon < h(\nu) +  \int\dfrac{\phi_n}{n}{d\nu}.
$$
Then we find $N_0 > 0$ such that $\nu \in {\mathcal M}_N$ for every $N \geq N_0$ and
$$
\forall \ N \geq N_0: \quad P\left(\dfrac{\phi_n}{n}\right) - \epsilon < \sup_{\nu \in {\mathcal M}_N}\left\{h(\nu) +  \int\dfrac{\phi_n}{n}d\nu \right\} \leq P\left(\dfrac{\phi_n}{n}\right).
$$
Therefore,
$$
P\left(\dfrac{\phi_n}{n}\right) - \epsilon < \sup_{N > 0}P_{N}\left(\dfrac{\phi_n}{n}\right) \leq P\left(\dfrac{\phi_n}{n}\right).
$$
This proves (\ref{pressure.sub.6}) since $\epsilon > 0$ is arbitrary.

The proof for the variational pressure $P^*(\Phi)$ is completely similar.
\end{proof}

To prove Lemma \ref{lemma.1.4} we need the following

\begin{lemma}\label{lemma.1.5}
There exists $N_0 >> 1$ large enough such that,
\begin{equation}\label{1.5}
\left|P_N\left(\dfrac{\phi_m}{m}\right) - P_M\left(\dfrac{\phi_m}{m}\right)\right| < 3\epsilon, \quad\forall m \geq \dfrac{nL}{\epsilon}, \quad\forall \ n \geq M \geq N \geq N_0.
\end{equation} 
\end{lemma}
\begin{proof}
By (\ref{definition.Phi.1}), (\ref{definition.M_N}) and (\ref{1.1.1}) we have that
\begin{equation}
\int\Phi{d\nu} \leq \int\dfrac{\phi_m}{m}d\nu < \int\Phi{d\nu} + 2\epsilon \quad\forall m \geq \dfrac{nL}{\epsilon}, \quad\forall \ n \geq N, \quad\forall \ \nu \in {\mathcal M}_N.
\end{equation}
Therefore, for every $N > 0$,
\begin{equation}\label{6}
P_N(\Phi) \leq P_N\left(\dfrac{\phi_m}{m}\right) \leq P_N(\Phi) + 2\epsilon \quad\forall m \geq \dfrac{nL}{\epsilon}, \quad\forall \ n \geq N.
\end{equation}
By Lemma \ref{lemma.2}
$$
P^*(\Phi) = \sup_{N > 0}P_N(\Phi).
$$
Then. we can choose $N_0 > 0$ such that
\begin{equation}\label{7}
|P_N(\Phi)-P_M(\Phi)| < \epsilon, \quad\forall \ M \geq N \geq N_0. 
\end{equation}
From (\ref{6}) and (\ref{7}), we get that
$$
0 \leq P_M\left(\dfrac{\phi_m}{m}\right) - P_N\left(\dfrac{\phi_m}{m}\right) < 3\epsilon, \quad\forall m \geq \dfrac{nL}{\epsilon}, \quad\forall \ n \geq M \geq N \geq N_0,
$$
so proving the Lemma.
\end{proof}

\begin{proof}[Proof of Lemma \ref{lemma.1.4}]
First notice that
$$
\left|P_N\left(\dfrac{\phi_n}{n}\right) - P_M\left(\dfrac{\phi_m}{m}\right)\right|  \leq  \left|P_N\left(\dfrac{\phi_n}{n}\right) - P_N\left(\dfrac{\phi_m}{m}\right)\right| + \left|P_N\left(\dfrac{\phi_m}{m}\right) - P_M\left(\dfrac{\phi_m}{m}\right)\right|.
$$
By (\ref{1.1.1}), in the proof of Lemma \ref{lemma.1.1},
\begin{equation}\label{4}
\forall \ n > 0:  \quad\left|P_N\left(\dfrac{\phi_n}{n}\right) - P_N\left(\dfrac{\phi_m}{m}\right)\right| < \epsilon, \quad\forall \ N > 0, \ \forall \ m \geq \dfrac{nL}{\epsilon}.
\end{equation}
We thus conclude, using (\ref{1.5}) in Lemma \ref{lemma.1.5}, that
\begin{equation}\label{8}
\left|P_N\left(\dfrac{\phi_n}{n}\right) - P_M\left(\dfrac{\phi_m}{m}\right)\right| < 4\epsilon, \quad\forall \ m \geq \dfrac{nL}{\epsilon}, \quad \forall \ n \geq M \geq N \geq N_0,
\end{equation}
This proves that
$$
\left\{P_N\left(\dfrac{\phi_n}{n}\right)\right\}_{n,N}
$$
is a Cauchy sequence.
\end{proof}

\end{document}